\documentclass[11pt]{article}
\usepackage{amsmath,amssymb,amsthm,color}

\usepackage{comment}
\usepackage{color}
\usepackage{amssymb}
\usepackage{amsmath}
\usepackage[english]{babel}
\usepackage{fancyhdr}
\usepackage{amsmath}

\def\ci{\begin{color}{red}\,}
\def\cf{\end{color}\,}

\newtheorem{thm}{Theorem}
\newtheorem{corollary}[thm]{Corollary}

\newtheorem{lemma}[thm]{Lemma}

\newtheorem{definition}[thm]{Definition}
\newtheorem{theorem}[thm]{Theorem}
\newtheorem{remark}[thm]{Remark}

\def\D{{\mathfrak D}\, }
\def\R{{\mathfrak R}\, }

\begin{document}
\begin{center}
{\bf Commuting maps on alternative rings}\footnote{
This work was supported by  FAPESP 19/03655-4; CNPq 302980/2019-9;  RFBR 20-01-00030. }
\\
\vspace{.2in}
{\bf Bruno Leonardo Macedo Ferreira}
\\
Federal University of Technology,\\
Professora Laura Pacheco Bastos Avenue, 800,\\
85053-510, Guarapuava, Brazil.
\\
brunoferreira@utfpr.edu.br
\\
\vspace{.2in}
{\bf Ivan Kaygorodov}
\\
Federal University of ABC,\\
dos Estados Avenue, 5001,\\
09210-580, Santo Andr\'{e}, Brazil. 
\\
kaygorodov.ivan@gmail.com
\\

\end{center}
\\
AMS: 17D05, 47B47.

\begin{abstract}
Suppose $\R$ is a $2$,$3$-torsion free unital alternative ring having an idempotent element $e_1$ $\left(e_2 = 1-e_1\right)$ which satisfies $x \R \cdot e_i = \{0\} \Rightarrow x = 0$ $\left(i = 1,2\right)$. In this paper, we aim to characterize the commuting maps. Let $\varphi$ be a commuting map of $\R$ so it is shown that $\varphi(x) = zx + \Xi(x)$ for all $x \in \R$, where $z \in \mathcal{Z}(\R)$ and $\Xi$ is an additive map from $\R$ into $\mathcal{Z}(\R)$. As a consequence a characterization of anti-commuting maps is obtained and we provide as an application, a characterization of commuting maps on von Neumann algebras relative alternative $C^{*}$-algebra with no central summands of type $I_1$.
\end{abstract}

\vspace{0.5 in}

\section*{Introduction}

Let $\mathfrak{R}$ be a unital ring not necessarily associative or commutative and consider the following convention for its multiplication operation: $xy\cdot z = (xy)z$ and $x\cdot yz = x(yz)$ for $x,y,z\in \mathfrak{R}$, to reduce the number of parentheses. We denote the {\it associator} of $\mathfrak{R}$ by $(x,y,z)=xy\cdot z-x\cdot yz$ for $x,y,z\in \mathfrak{R}$. And $[x,y] = xy - yx$ is the usual Lie product of $x$ and $y$, with $x,y \in \mathfrak{R}$.

Let be $\varphi:\mathfrak{R}\rightarrow \mathfrak{R}$ a additive map of $\mathfrak{R}$ into $\mathfrak{R}$. We call $\varphi$ a {\it commuting map} of $\mathfrak{R}$ into $\mathfrak{R}$ if for all $x \in \mathfrak{R}:$ 
\[  
[\varphi(x),x] = 0.\]
Divinsky at \cite{Divi} started a study on commuting maps, he proved that a simple Artinian ring is commutative if it has a commuting automorphism different from the identity map. In the face of this study, some mathematicians began to investigate the problem of characterizing some maps on associative rings.  Bresar in \cite{bresar2} proved that, if $F$ is a commuting additive map from a von Neumann algebra $M$ into itself, then there exist $Z \in \mathcal{Z}(M)$ and an additive map $h : M \rightarrow \mathcal{Z}(M)$ such that $F(A) = ZA + h(A)$ for all $A \in M$. Later, Bresar \cite{bresar} gave the same characterization of commuting additive maps on prime rings. For the reader interested in other results about characterization of maps on associative rings and algebras, see refs. \cite{bresar3, bresar4, cheung, Lee, changd} and the references therein. For the class of  alternative rings and algebras studies about many linear and non-linear maps have been an interesting and active research topic recently, we can quote \cite{Fer, bruth, posd, FerGur, fgk20,kp, FerGuFe}. An important class of $8$-dimensional Cayley algebras (or Cayley-Dickson algebras, the prototype having been
discovered in 1845 by Cayley and later generalized by Dickson) is so called \textit{octonions algebra} a class of alternative algebras which are not
associative. 
Given any algebra $\mathfrak{A}$ of dimension $n$ with an $*$-involution an algebra $\mathfrak{B}$ of dimension $2n$ was constructed by Albert \cite{Albert} as $\mathfrak{B} = \left\{x + y i ~ | ~ x,y \in \mathfrak{A} \right\}$. Albert proved that if $\mathfrak{B}$ is associative, then $\mathfrak{A}$ is commutative and that $\mathfrak{A}$ is associative if and only if the elements of $\mathfrak{B}$ satisfy $(u,u,v)=0=(v,u,u)$. Given this we have the following. 

A ring $\mathfrak{R}$ is said to be {\it alternative} if $(x,x,y)=0=(y,x,x)$ for all $x,y\in \mathfrak{R}$. One easily sees that any associative ring is an alternative ring. 
A  ring $\mathfrak{R}$ is called {\it k-torsion free} if $k\,x=0$ implies $x=0,$ for any $x\in \mathfrak{R},$ where $k\in{\mathbb Z},\, k>0$, and {\it prime} if $\mathfrak{AB} \neq 0$ for any two nonzero ideals $\mathfrak{A},\mathfrak{B}\subseteq \mathfrak{R}$.
The {\it nucleus} of a ring $\mathfrak{R}$ \mbox{is defined by} $$\mathcal{N}(\mathfrak{R})=\{r\in \mathfrak{R}\mid (x,y,r)=(x,r,y)=(r,x,y)=0 \hbox{ for all }x,y\in \mathfrak{R}\}.$$
The {\it commutative center} of an ring $\mathfrak{R}$ \mbox{is defined by} $$\mathcal{Z}(\mathfrak{R})=\{r\in \mathcal{\R}\mid [r, x] = 0 \hbox{ for all }x \in \mathfrak{R}\}.$$

\vspace{.2in}
The next result can be found in \cite{bruth}
\begin{theorem}\label{meu}
Let $\mathfrak{R}$ be a $3$-torsion free alternative ring. So
$\mathfrak{R}$ is a prime ring if and only if $a\mathfrak{R} \cdot b=0$ (or $a \cdot \mathfrak{R}b=0$) implies $a = 0$ or $b =0$ for $a, b \in \mathfrak{R}$. 
\end{theorem}

\vspace{.2in}
A nonzero element $e_{1}\in \mathfrak{R}$ is called an {\it idempotent} if $e_{1}e_{1}=e_{1}$ and a {\it nontrivial idempotent} if it is an idempotent different from the multiplicative identity element of $\mathfrak{R}$. Let us consider $\mathfrak{R}$ an alternative ring and fix a nontrivial idempotent $e_{1}\in\mathfrak{R}$. Let \mbox{$e_2 \colon\mathfrak{R}\rightarrow\mathfrak{R}$} and $e'_2 \colon\mathfrak{R}\rightarrow\mathfrak{R}$ be linear operators given by $e_2(a)=a-e_1a$ and $e_2'(a)=a-ae_1.$ Clearly $e_2^2=e_2,$ $(e_2')^2=e_2'$ and we note that if $\mathfrak{R}$ has a unity, then we can consider $e_2=1-e_1\in \mathfrak{R}$. Let us denote $e_2(a)$ by $e_2a$ and $e_2'(a)$ by $ae_2$. It is easy to see that $e_ia\cdot e_j=e_i\cdot ae_j~(i,j=1,2)$ for all $a\in \mathfrak{R}$. Then $\mathfrak{R}$ has a Peirce decomposition
$\mathfrak{R}=\mathfrak{R}_{11}\oplus \mathfrak{R}_{12}\oplus
\mathfrak{R}_{21}\oplus \mathfrak{R}_{22},$ where
$\mathfrak{R}_{ij}=e_{i}\mathfrak{R}e_{j}$ $(i,j=1,2)$ \cite{He}, satisfying the following multiplicative relations:
\begin{enumerate}\label{asquatro}
\item [\it (i)] $\mathfrak{R}_{ij}\mathfrak{R}_{jl}\subseteq\mathfrak{R}_{il}\
(i,j,l=1,2);$
\item [\it (ii)] $\mathfrak{R}_{ij}\mathfrak{R}_{ij}\subseteq \mathfrak{R}_{ji}\
(i,j=1,2);$
\item [\it (iii)] $\mathfrak{R}_{ij}\mathfrak{R}_{kl}=0,$ if $j\neq k$ and
$(i,j)\neq (k,l),\ (i,j,k,l=1,2);$
\item [\it (iv)] $x_{ij}^{2}=0,$ for all $x_{ij}\in \mathfrak{R}_{ij}\ (i,j=1,2;~i\neq j).$
\end{enumerate}

The notion of Peirce decomposition for alternative rings is similar to that one for associative rings. However, this similarity is restricted to its written form, not including its theoretical structure since Peirce decomposition for alternative rings is a generalization of that classical one for associative rings. Taking this fact into account, in the present paper in a certain way we generalize the main Bresar's Theorem \cite{bresar} to the class of alternative rings.

As mentioned before, recently the problem of characterizing maps on nonassociative rings has been studied. Ferreira and Guzzo worked the characterization of multiplicative Lie derivation on alternative rings. 

In \cite{posd}, they investigated the additivity of Lie multiplicative map where did they get the
following result. 

\begin{theorem}\label{FerGur} Let $\mathfrak{R}$ and $\mathfrak{R}'$ be alternative rings.
Suppose that $\mathfrak{R}$ is a ring containing a nontrivial idempotent $e_1$ which satisfies:
\begin{enumerate}
\item[\it (i)] If $[a_{11}+ a_{22}, \mathfrak{R}_{12}] = 0$, then $a_{11} + a_{22} \in \mathcal{Z}(\mathfrak{R}),$
\item[\it (ii)] If $[a_{11}+ a_{22}, \mathfrak{R}_{21}] = 0$, then $a_{11} + a_{22} \in \mathcal{Z}(\mathfrak{R}).$
\end{enumerate}
Then every Lie multiplicative bijection $\varphi$ of $\mathfrak{R}$ onto an arbitrary alternative ring $\mathfrak{R}'$ is almost additive.
\end{theorem}

In another recent paper \cite{FerGur} they proved 

\begin{theorem} \label{Fegu2}
Let $\R$ be a unital $2$,$3$-torsion free alternative ring with nontrivial idempotents $e_1$, $e_2$ and with associated 
Peirce decomposition $\R = \R_{11} \oplus \R_{12} \oplus \R_{21} \oplus \R_{22}$. Suppose that $\mathfrak{R}$ satisfies the following conditions:
\begin{enumerate}
\item [{\rm (1)}] If $x_{ij}\R_{ji} = 0$, then $x_{ij} = 0$ ($i \neq j$);
\item [{\rm (2)}] If $x_{11}\R_{12} = 0$ or $\R_{21}x_{11} = 0$, then $x_{11} = 0$;
\item [{\rm (3)}] If $\R_{12}x_{22} = 0$ or $x_{22}\R_{21} = 0$, then $x_{22} = 0$;
\item [{\rm (4)}] If $z \in \mathcal{Z}(\mathfrak{R})$ with $z \neq 0$, then $z\R = \R$.
\end{enumerate}
Let $\D \colon  \R \longrightarrow \R$ be a multiplicative Lie derivation of $\mathfrak{R}$. 
Then $\D$ is the form $\delta + \tau$, where $\delta$ is an additive derivation of $\R$ and $\tau$ is a mapping from 
$\R$ into the commutative centre $\mathcal{Z}(\mathfrak{R})$,  which maps commutators into the zero if and only if
\begin{enumerate}
\item[(a)] $e_2\D(\R_{11})e_2 \subseteq \mathcal{Z}(\R) e_2,$
\item[(b)] $e_1\D(\R_{22})e_1 \subseteq \mathcal{Z}(\R) e_1.$
\end{enumerate}
\end{theorem}

Inspired by the above mentioned results the purpose of the present paper is to consider the problem of characterizing commuting additive maps on alternative rings.

Let $\R$ be a unity alternative ring and an idempotent element $e_1$, and let $\mathcal{Z}(\R)$ denote the commutative center of $\R$. Assume that the characteristic of $\R$ is not $2$, $3$ and satisfies $x \R \cdot e_i = {0} \Rightarrow x = 0$ $\left(i = 1,2\right)$. Let $\varphi : \R \rightarrow \R$ be an additive map. We show that the following two statements are equivalent: 
\begin{enumerate}
    \item  $\varphi$ is commuting; 
    \item   there exist $z \in \mathcal{Z}(\R)$ and an additive map $\Xi : \R \rightarrow \mathcal{Z}(\R)$ such that $\varphi(x) = zx + \Xi(x)$ for all $x \in \R$. 
\end{enumerate}As applications, a characterization of commuting additive maps on prime rings and von Neumann algebras relative alternative $C^{*}$-algebra with no central summands of type $I_1$ is obtained.

\section{Main theorem}

We shall prove as follows the main result of this paper.

\begin{theorem}\label{mainthm} 
Let $\R$ be a unital $2$,$3$-torsion free alternative ring. Assume that $\R$ has a nontrivial idempotent $e_1$ with associated Peirce decomposition $\R = \R_{11} \oplus \R_{12} \oplus \R_{21} \oplus \R_{22}$, such that $x \R \cdot e_i = {0} \Rightarrow x = 0$ $\left(i = 1,2\right)$. 
Let $\varphi : \R \rightarrow \R$ be an additive map. Then the following statements are equivalent:
\begin{enumerate}
\item[$(\spadesuit)$] $\varphi$ is commuting;
\item [$(\clubsuit)$] There exist $z \in \mathcal{Z}(\R)$ and an additive map $\Xi : \R \rightarrow \mathcal{Z}(\R)$ such that $\varphi(x) = zx + \Xi(x)$ for all $x \in \R$.
\end{enumerate}
\end{theorem}
 It is obvious that prime alternative rings satisfy the assumptions ``$x \R \cdot e_1 = {0} \Rightarrow x = 0$ and $x \R \cdot e_2 = {0} \Rightarrow x = 0$"
by Theorem \ref{meu}. So, we have the following result

\begin{corollary}
Let $\R$ be a unital $2$,$3$-torsion free prime alternative ring. Assume that $\R$ has a nontrivial idempotent $e_1$ with associated Peirce decomposition $\R = \R_{11} \oplus \R_{12} \oplus \R_{21} \oplus \R_{22}$. Let $\varphi : \R \rightarrow \R$ be an additive map. Then the following statements are equivalen:
\begin{enumerate}
\item[$(\spadesuit)$] $\varphi$ is commuting;
\item [$(\clubsuit)$] There exist $z \in \mathcal{Z}(\R)$ and an additive map $\Xi : \R \rightarrow \mathcal{Z}(\R)$ such that $\varphi(x) = zx + \Xi(x)$ for all $x \in \R$.
\end{enumerate}
\end{corollary}

 \begin{remark}
Some examples of non-prime alternative rings satisfying  the assumptions ``$x \R \cdot e_1 = {0} \Rightarrow x = 0$ and $x \R \cdot e_2 = {0} \Rightarrow x = 0$"
were given in \cite{posd}. 
\end{remark}

\vspace{.1in}
\section{The proof of main result}

It is clear that `` $(\clubsuit) \Rightarrow (\spadesuit)$ ". We will prove ``$(\spadesuit) \Rightarrow (\clubsuit)$"  by a series of Lemmas.
The following Lemmas has the same hypotheses of Theorem \ref{mainthm} and we need these Lemmas for the proof of ``$(\spadesuit) \Rightarrow (\clubsuit)$". Thus, let us consider $e_{1}$ a nontrivial idempotent of $\mathfrak{R}$. 

We started with the following Lemma that characterize the commutative center of  an alternative ring.

\begin{lemma}\label{lema1}
\begin{eqnarray*}
\mathcal{Z}(\R) &=&  \{z_{11} + z_{22} : z_{11} \in \R_{11}, z_{22} \in \R_{22},\\&& \left[z_{11} + z_{22}, \R_{12}\right] = \left[z_{11} + z_{22}, \R_{21}\right] = \{0\} \}
\end{eqnarray*}
\end{lemma}
\begin{proof}
 On the one hand assume that $z = z_{11} + z_{12} + z_{21} + z_{22} \in \mathcal{Z}(\R)$. Then $ze_1 = e_1 z$ implies $z_{12} =
z_{21} = 0$. Furthermore, for any $x_{12} \in \R_{12}$ and $x_{21} \in \R_{21}$, it follows that $zx_{12} = x_{12}z$ and
$zx_{21} = x_{21}z$ that 
\[\left[z_{11} + z_{22}, \R_{12}\right] = \left[z_{11} + z_{22}, \R_{21}\right] = \{0\}.\]

On the other hand, assume that $z_{11} \in \R_{11}, z_{22} \in \R_{22}$, and  
\[\left[z_{11} + z_{22}, \R_{12}\right] = \left[z_{11} + z_{22}, \R_{21}\right] = \{0\}.\] To prove $z_{11} + z_{22} \in \mathcal{Z}(\R)$, one only needs to check $z_{ii} \in \mathcal{Z}(\R_{ii}), i = 1, 2$. In fact, for any $r_{11} \in \R_{11}$ and any $r_{12} \in \R_{12}$, we have 
\[
(z_{11}r_{11} - r_{11}z_{11})r_{12}  = (z_{11}r_{11})r_{12} - (r_{11}z_{11})r_{12}  = z_{11}(r_{11}r_{12}) - r_{11}(z_{11}r_{12}) =\]
\[(r_{11}r_{12})z_{22} - r_{11}(r_{12}z_{22}) =r_{11}(r_{12}z_{22}) - r_{11}(r_{12}z_{22}) =  0.\]
Hence $(z_{11}r_{11} - r_{11}z_{11})\R \cdot e_2 = {0}$. Therefore $z_{11} \in \mathcal{Z}(\R_{11})$ by condition of Theorem \ref{mainthm}. Similarly, we can check $z_{22} \in \mathcal{Z}(\R_{22})$.  
\end{proof}

\vspace{.1in}
Now we give the following lemma which plays a crucial role in this paper.

\begin{lemma}\label{crucial}
For $z_{ii} \in \mathcal{Z}(\R_{ii})$, $i = 1, 2$, there exists an element $z \in \mathcal{Z}(\R)$ such that $z_{ii} = ze_i$.
\end{lemma}

\begin{proof}
Since $\R$ is 3-torsion free alternative ring we get $\mathcal{Z}(\R_{ii}) \subseteq \mathcal{N}(\R_{ii})$. Let be $z_{ii} \in \mathcal{Z}(R_{ii})$, it is clear that $e_ixz_{ii} = z_{ii}xe_i$ holds for all $x \in \R$ then, by  \cite[Lemma $4$]{Bai}, there is an element $z \in \mathcal{Z}(\R)$ such that $z_{ii} = ze_i$.  
\end{proof}

\vspace{.1in}
The next Lemma shows where idempotents of $\R$ are taken by a commuting map.

\begin{lemma}\label{lema2}
$\left\{\varphi(1), \varphi(e_i) \right\} \subset \R_{11} + \R_{22}$. Moreover, there exist some $z_i \in \mathcal{Z}(\R)$ such that $e_i\varphi(1)e_i = z_i e_i$ with $i \in \{1,2\}$.
\end{lemma}

\begin{proof}
For any $x \in \R$, we have $[\varphi(x + 1), x + 1] = 0$. Hence 
\[0 = [\varphi(x + 1), x + 1] = [\varphi(x + 1), x] = [\varphi(1), x]\] holds for all $x \in \R$. Particularly, $[\varphi(1), e_1] = 0$ and $\varphi(1) \in \R_{11} + \R_{22}$. As $0 = [\varphi(1), x_{ii}] = [e_i\varphi(1)e_i, x_{ii}]$ for all $x_{ii} \in \R_{ii}$. These imply $e_i\varphi(1)e_i \in \mathcal{Z}(\R_{ii})$, that is, there exists $z_i \in  \mathcal{Z}(\R)$ such that $e_i\varphi(1)e_i = z_ie_i$ with $i \in \{1,2\}$. For $e_i$, note that $[\varphi(e_i), e_i] = 0$. It follows that $\varphi(e_i) \in \R_{11} + \R_{22}$.
\end{proof}

\vspace{.1in}
The next Lemmas \ref{lema3}-\ref{lema4} we setting the components of Peirce's decomposition are hold by commuting map.

\begin{lemma}\label{lema3}
For any $x_{ii} \in \R_{ii}$, we have $\varphi(x_{ii}) \in \R_{11}+\R_{22}$ and there exists $z_i \in \mathcal{Z}(\R)$ such that $e_j \varphi(x_{ii})e_j = z_i e_j$, $1\leq i \neq j \leq 2$. Consequently, $e_j \varphi(e_i)e_j = z_i e_j$.
\end{lemma}

\begin{proof}
Taking any $x_{11} \in \R_{11}$, we have 
\[0 = [\varphi(x_{11}), x_{11}] = [e_1\varphi(x_{11})e_1 + e_1\varphi(x_{11})e_2 + e_2\varphi(x_{11})e_1 + e_2\varphi(x_{11})e_2, x_{11}] =\] \[ [ e_1\varphi(x_{11})e_1, x_{11}] \linebreak + [e_1\varphi(x_{11})e_2, x_{11}] + [e_2\varphi(x_{11})e_1, x_{11}] \in \R_{11} + \R_{12} + \R_{21}.\]
Thus,
$$[e_1\varphi(x_{11})e_1, x_{11}] = [e_1\varphi(x_{11})e_2, x_{11}] = [e_2\varphi(x_{11})e_1, x_{11}]=0$$
holds for any $x_{11} \in \R_{11}$. It follows that $(x_{11} + e_1)e_1\varphi(x_{11}+ e_1)e_2 = 0.$ Since $e_1\varphi(e_1)e_2 = 0$ we get $e_1\varphi(x_{11})e_2 = 0$. As $[e_2\varphi(x_{11})e_1, x_{11}] = 0$ using a similar argument to the above we have that $e_2\varphi(x_{11})e_1 = 0$.
Thus, we have proved that $\varphi(x_{11}) \in \R_{11}+\R_{22}$ and  $[e_1\varphi(x_{11})e_1, x_{11}] = 0$ holds for all $x_{11} \in \R_{11}$.
Analogously, we get $\varphi(x_{22}) \in \R_{11}+\R_{22}$ and  $[e_2\varphi(x_{22})e_2, x_{22}] = 0$ holds for all $x_{22} \in \R_{22}$.
Now note that 
\[0= [\varphi(x_{11}+x_{22}), x_{11}+x_{22}] = [e_2\varphi(x_{11})e_2, x_{22}] + [e_1\varphi(x_{22})e_1, x_{11}],\] and so 
$$[e_2\varphi(x_{11})e_2, x_{22}] = [e_1\varphi(x_{22})e_1, x_{11}]=0.$$
It follows that $e_j \varphi(x_{ii}) e_j \in \mathcal{Z}(\R_{jj})$ and by Lemma \ref{crucial} we obtain $e_j \varphi(x_{ii}) e_j = z_i e_j$, where $z_i \in \mathcal{Z}(\R)$,  $1\leq i \neq j \leq 2$.     

\end{proof}

\begin{lemma}\label{lema4}
For any $x_{ij} \in \R_{ij}$ with $1\leq i \neq j \leq 2$ we have
\begin{enumerate}
	\item [\it (i)] $e_j \varphi(x_{ij})e_i = 0$;
	\item [\it (ii)] $e_i \varphi(x_{ij})e_j = [e_i\varphi(e_i)e_i + e_j\varphi(e_i)e_j, x_{ij}] = -[e_j\varphi(e_j)e_j + e_i\varphi(e_j)e_i, x_{ij}] $; 
	\item [\it (iii)] $[e_i \varphi(x_{ij})e_i + e_j \varphi(x_{ij})e_j, x_{ij}] = 0$;
	\item [\it (iv)] $e_i\varphi(x_{ij})e_i = z e_i$ and $e_j\varphi(x_{ij})e_j = z'e_j$ with $z,z' \in \mathcal{Z}(\R)$.
\end{enumerate}
\end{lemma}

\begin{proof}
We will only prove case $i=1, j=2$ because the other the case $i=2, j=1$ has similar proof. 
For any $x_{12} \in \R_{12}$ we have 
\[0=[\varphi(x_{12}), x_{12}] = e_1\varphi(x_{12})e_1x_{12} + e_1\varphi(x_{12})e_2x_{12} + e_2\varphi(x_{12})e_1x_{12}\] \[ - x_{12}e_1\varphi(x_{12})e_2- x_{12}e_2\varphi(x_{12})e_1 - x_{12}e_2\varphi(x_{12})e_2.\] Hence we get $[e_1\varphi(x_{12})e_1 + e_2\varphi(x_{12})e_2, x_{12}] = 0$ this shows $(iii)$ and 
\[x_{12}e_2\varphi(x_{12})e_1 = 0, \ e_1\varphi(x_{12})e_2 x_{12} = 0,  \  e_2\varphi(x_{12})e_1 x_{12} = 0.\] Now we have too 
\[0= [\varphi(x_{12} + e_1), x_{12}+e_1] =\] 
\[e_2\varphi(x_{12})e_1 - e_1\varphi(x_{12})e_2 + e_1\varphi(e_{1})e_1x_{12} - x_{12}e_2\varphi(e_{1})e_2.\] 
Thus, $e_2 \varphi(x_{12})e_1 = 0$ that is $(i)$ and 
\begin{eqnarray}\label{pri}
e_1 \varphi(x_{12})e_2 = [e_1\varphi(e_1)e_1 + e_2\varphi(e_1)e_2, x_{12}].
\end{eqnarray} 
As $[\varphi(x_{12} + e_2) , x_{12} + e_2] = 0$ by a straightforward calculation follows that 
\begin{eqnarray}\label{seg}
e_1 \varphi(x_{12})e_2 = -[e_2\varphi(e_2)e_2 + e_1\varphi(e_2)e_1, x_{12}] 
\end{eqnarray} 
Of (\ref{pri}) and (\ref{seg}) we obtain $(ii)$. We still need to proof $(iv)$. For this to observe that $[e_1 \varphi(x_{11})e_1, x_{11}] =0$ by the Lemma \ref{lema3}. Since \[[ \varphi(x_{11} + x_{12}) , x_{11} + x_{12}] = 0 \mbox{ and } [\varphi(x_{12}), x_{11}] \in \R_{11} + \R_{12}\] and using the identities $(i)$ and $(iii)$ we conclude \[[\varphi(x_{12}), x_{11}] = 0 \mbox{ and }[e_1\varphi(x_{12})e_1, x_{11}] = 0.\] Therefore, by Lemma \ref{crucial} there exists $z \in \mathcal{Z}(\R)$ such that $e_1\varphi(x_{12})e_1 = z e_1$ that is $(iv)$. The proof is complete.    
\end{proof}

In addition to Lemma \ref{lema2} the next Lemma we showed specifically where commuting map takes the unity of $\R$. 

\begin{lemma}\label{lema5}
$\varphi(1) \in \mathcal{Z}(\R)$, and, there exists $z_i \in \mathcal{Z}(\R)$ such that $e_i \varphi(e_i)e_i = z_i e_i \in
\mathcal{Z}(\R_{ii})$, $i = 1, 2$.
\end{lemma}

\begin{proof}
By Lemmas \ref{lema2} and \ref{lema3}, we have
$$e_1 \varphi(e_1) e_1 = e_1 \varphi(1) e_1 - e_1 \varphi(e_2) e_1 = (\alpha - \beta)e_1 = z_1 e_1 \in \mathcal{Z}(\R_{11})$$and
$$e_2 \varphi(e_2) e_2 = e_2 \varphi(1) e_2 - e_2 \varphi(e_1) e_2 = (\lambda - \mu)e_1 = z_2 e_1 \in \mathcal{Z}(\R_{22}),$$
where $\alpha , \beta ,\lambda , \mu \in \mathcal{Z}(\R)$ and $(\alpha - \beta) = z_1$, $(\lambda - \mu) = z_2$.

By Lemma \ref{lema4} items $(ii)$ and $(iii)$ we get 
\[[e_1\varphi(e_1)e_1 + e_2\varphi(e_1)e_2, x_{12}] = -[e_2\varphi(e_2)e_2 + e_1\varphi(e_2)e_1, x_{12}],\] hence
$[\varphi(1), x_{12}] = [\varphi(e_1) + \varphi(e_2), x_{12}] = 0$. Similarly, $[\varphi(1), x_{21}] = 0$. It follows that $\varphi(1) \in \mathcal{Z}(\R)$ by Lemma \ref{lema1}. 

\end{proof}

\vspace{.1in}
And finally, the next Lemmas \ref{lema6}-\ref{lema9} tells us the behavior to the image of Peirce's components by commuting map.

\begin{lemma}\label{lema6}
There exists $z, z' \in \mathcal{Z}(\R)$ such that $e_1\varphi(e_1)e_1 + e_2\varphi(e_2)e_2 - (ze_1 + z' e_2) \in \mathcal{Z}(\R)$.
\end{lemma}

\begin{proof}
By Lemma \ref{lema5} we have 
\[e_1\varphi(e_1)e_1 - ze_1 \in \mathcal{Z}(\R_{11}) \mbox{ and }e_2\varphi(e_2)e_2 - z' e_2 \in \mathcal{Z}(\R_{22}).\] Moreover, since $\varphi(1) \in \mathcal{Z}(\R)$ and Lemma \ref{lema3} we get 
\[e_1\varphi(e_1)e_1 x_{12} + x_{12}(z' e_2) = (e_1\varphi(e_1)e_1 + z' e_2)x_{12} =\] \[(e_1\varphi(e_1)e_1 + e_1\varphi(e_2) e_1)x_{12} =x_{12}(e_2\varphi(e_1)e_2 + e_2\varphi(e_2) e_2) =\]
\[x_{12}(z e_2  +  e_2\varphi(e_2) e_2 ) = ze_1 x_{12} + x_{12} e_2\varphi(e_2) e_2,\] that is, $[(e_1\varphi(e_1)e_1  - ze_1)+ (e_2\varphi(e_2) e_2 -z' e_2), x_{12}] = 0$ for all $x_{12} \in \R_{12}$. Similarly, $[(e_1\varphi(e_1)e_1  - ze_1)+ (e_2\varphi(e_2) e_2 -z' e_2), x_{21}] = 0$ for all $x_{21} \in \R_{21}$. By Lemma \ref{lema1} we conclude $(e_1\varphi(e_1)e_1  - ze_1)+ (e_2\varphi(e_2) e_2 -z' e_2) \in \mathcal{Z}(\R)$.
  
\end{proof}

\begin{lemma}\label{lema7}
We have $[e_1 \varphi(x_{ij})e_1 + e_2 \varphi(x_{ij})e_2, x_{ji}] = 0$ for all $x_{ij} \in \R_{ij}$ and $x_{ji} \in \R_{ji}$ with $i \neq j$.
\end{lemma}

\begin{proof}
By definition of $\varphi$ we have $[\varphi(1 + x_{12} + x_{21}), 1 + x_{12} + x_{21}] = 0$. Using Lemmas \ref{lema4} and \ref{lema5} we get
\[[\varphi(1 + x_{12} + x_{21}), 1 + x_{12} + x_{21}] = [\varphi(x_{12}), x_{21}] + [\varphi(x_{21}), x_{12}],\] which implies that 
\[[e_1 \varphi(x_{12})e_1 + e_2 \varphi(x_{12})e_2, x_{21}] = [e_1 \varphi(x_{21})e_1 + e_2 \varphi(x_{21})e_2, x_{12}] = 0\] for all $x_{12} \in \R_{12}$ and $x_{21} \in \R_{21}$.
\end{proof}

\begin{lemma}\label{lema8}
We have $[e_i \varphi(x_{ij})e_i + e_j \varphi(x_{ij})e_j, r] = 0$ for all $r \in \R$ and $x_{ij} \in \R_{ij}$ with $i \neq j$.
\end{lemma}

\begin{proof}
By Lemma \ref{lema4} item $(iv)$ we have  $e_i \varphi(x_{ij})e_i = ze_i$ and  \linebreak $e_j \varphi(x_{ij})e_j = z' e_j$, where $z, z' \in \mathcal{Z}(\R)$. and using Lemma \ref{lema7} we get $[ze_i + z'e_j , x_{ji}] = 0$, that implies 
$$(z'e_j - ze_j)x \cdot e_i = 0,$$
for all $x \in \R$. Thus by assumption of Theorem \ref{mainthm} we have $ze_j = z'e_j$. Therefore 
\[[e_i \varphi(x_{ij})e_i + e_j \varphi(x_{ij})e_j, r] = [ze_i + z'e_j, r] = [ze_i + ze_j, r] = [z, r] = 0\] for all $r \in \R$.

\end{proof}

\begin{lemma}\label{lema9}
We have $e_i \varphi(x_{ii})e_i = z_ie_i + (e_i\varphi(e_i)e_i - z'_ie_i)x_{ii}$ for all $x_{ii} \in \R_{ii}$ with $z_i,z'_i \in \mathcal{Z}(\R)$ and $i \in \left\{1,2\right\}$.
\end{lemma}

\begin{proof}
Consider $i \neq j$ with $i,j \in \left\{1,2\right\}$. Let be $x_{ii} \in \R_{ii}$ and $x_{ij} \in \R_{ij}$, by Lemma \ref{lema5} and Lemma \ref{lema4} item $(iii)$, we get
$$0 = [\varphi(x_{ii} + x_{ij}), x_{ii} + x_{ij}] = [\varphi(x_{ii}), x_{ij}] + [\varphi(x_{ij}), x_{ii}].$$ 
It follow that 
\begin{eqnarray}\label{doze}
e_i\varphi(x_{ii})e_i x_{ij} - x_{ij}e_j\varphi(x_{ii})e_j - x_{ii} e_i\varphi(x_{ij})e_j = 0,
\end{eqnarray}
where have used Lemma \ref{lema8} and Lemma \ref{lema4} item $(i)$. By Lemma \ref{lema3} to observe that $e_j \varphi(x_{ii})e_j = z_ie_j$ with $z_i = {z_i}_{11} + {z_i}_{22}  \in \mathcal{Z}(\R)$ since 
\[x_{ij} \cdot ({z_i}_{11} + {z_i}_{22})e_j = ({z_i}_{11} + {z_i}_{22})e_i \cdot x_{ij}\] and by Lemma \ref{lema4} item $(ii)$ we have 
\[x_{ii} e_i \varphi(x_{ij})e_j = x_{ii}\cdot e_i\varphi(e_i)e_i x_{ij} - x_{ii}\cdot z'_ie_i  x_{ij} = x_{ii} e_i\varphi(e_i)e_i \cdot x_{ij} - x_{ii}z'_ie_i \cdot x_{ij},\] with $z'_i \in \mathcal{Z}(\R)$. Hence we can write (\ref{doze}) as 
\[e_i\varphi(x_{ii})e_i x_{ij} - z_ie_i \cdot x_{ij} -  x_{ii} e_i\varphi(e_i)e_i \cdot x_{ij} + x_{ii}z'_ie_i \cdot x_{ij} = 0,\] 
that is,
$$(e_i\varphi(x_{ii})e_i  - z_ie_i -  x_{ii} e_i\varphi(e_i)e_i  + x_{ii}z'_ie_i) e_ix e_j =0 \ \ \mbox{for all}  \ \ x \in \R.$$ 
By the assumption of Theorem \ref{mainthm} we get 
\[e_i\varphi(x_{ii})e_i =  z_ie_i + (e_i\varphi(e_i)e_i  - z'_ie_i) x_{ii}.\]
\end{proof}

\vspace{.2in}

Now we are ready to prove ``$(\spadesuit) \Rightarrow (\clubsuit)$" of the Theorem \ref{mainthm}.

\vspace{.2in}

\noindent \textbf{Prove of the Theorem \ref{mainthm} ($(\spadesuit) \Rightarrow (\clubsuit)$):} 
Let us define $\Xi(x) = \varphi(x) - zx$ with $z = e_1\varphi(e_1)e_1 + e_2\varphi(e_2)e_2 - (z_1e_1 + z_2 e_2) \in \mathcal{Z}(\R)$ likewise Lemma \ref{lema6}. Note that $e_j\varphi(e_i)e_j = z_i e_j$, with $i \neq j$. It is clear that $\Xi$ is additive on $\R$. Now we just need to prove that $\Xi(x) \in \mathcal{Z}(\R)$. 
Let be $x = x_{11} + x_{12} + x_{21} + x_{22} \in \R$. By Lemmas \ref{lema3} and \ref{lema4} we have

\begin{eqnarray*}
\Xi(x) &=& \underbrace{e_1 \varphi(x_{11})e_1 + e_2 \varphi(x_{11})e_2 - e_1 \varphi(e_{1})e_1 + z_1 e_1 x_{11}}_{\in \mathcal{Z}(\R) \mbox{ \ by Lemma \ref{lema9}}}
\\&+& \underbrace{e_1 \varphi(x_{22})e_1 + e_2 \varphi(x_{22})e_2 - e_2 \varphi(e_{2})e_2x_{22} + z_2 e_2 x_{22}}_{\in \mathcal{Z}(\R) \mbox{ \ by Lemma \ref{lema9}}}
\\&+& \underbrace{e_1 \varphi(x_{12})e_1 + e_2 \varphi(x_{12})e_2 + e_1 \varphi(x_{21})e_1 + e_2 \varphi(x_{21})e_2}_{\in \mathcal{Z}(\R) \mbox{ \ by Lemma \ref{lema8}}} 
\\&+& \underbrace{z_1e_1 \cdot x_{12} - x_{12}e_2 \varphi(e_{1})e_2 + z_2e_2 \cdot x_{21} - x_{21}e_1 \varphi(e_{2})e_1}_{\in \mathcal{Z}(\R) \mbox{ \ by Lemma \ref{lema1}}}.
\end{eqnarray*} 
Therefore the proof $(\spadesuit) \Rightarrow (\clubsuit)$ of the Theorem \ref{mainthm} is complete.

\begin{definition}
A nonlinear map $\varphi : \R \rightarrow \R$ is called anti-commuting if $[\varphi(a),b] = -[a,\varphi(b)]$ holds for all $a,b \in \R$.
\end{definition}

\begin{corollary}
Let $\R$ be a unital $2$,$3$-torsion free alternative ring. Assume that $\R$ has a nontrivial idempotent $e_1$ with associated Peirce decomposition $\R = \R_{11} \oplus \R_{12} \oplus \R_{21} \oplus \R_{22}$, such that $x \R \cdot e_i = {0} \Rightarrow x = 0$  $\left(i = 1,2\right)$. 
Let $\varphi : \R \rightarrow \R$ be an additive map. Then the following statements are equivalent:
\begin{enumerate}
\item[$(\spadesuit)$] $\varphi$ is anti-commuting;
\item [$(\clubsuit)$] There exist $z \in \mathcal{Z}(\R)$ and an additive map $\Xi : \R \rightarrow \mathcal{Z}(\R)$ such that $\varphi(x) = zx + \Xi(x)$ for all $x \in \R$.
\end{enumerate}
\end{corollary}

\section{Application on alternative $C^{*}$-algebras}

Recall that an alternative $C^{*}$-algebra $\mathcal{U}$ is a complete normed alternative complex algebra endowed with a conjugate-linear algebra involution $*$ \linebreak satisfying
$\left\|a^* a\right\| = \left\|a\right\|^{2}$ for every $a \in \mathcal{U}$. It is well known that alternative $C^{*}$-algebra admits representation as a von Neumann algebra. Let be $\mathcal{M}_\mathcal{U}$ the von Neumann algebra relative alternative $C^{*}$-algebra $\mathcal{U}$.
For those readers who are not familiar with this language of alternative $C^{*}$-algebra we recommend \cite{Miguel1, Miguel2}.
It is shown \cite{Bai} that, if a von Neumann algebra $\mathcal{M}$ has no central summands of type $I_1$  $\left( = \mbox{central abelian projection} \right)$, then $\mathcal{M}$ satifies the follow assumptions 

\begin{itemize}
\item $X \mathcal{M} \cdot e_1 = \left\{0\right\} \Rightarrow X = 0$, 
\item $X \mathcal{M} \cdot e_2 = \left\{0\right\} \Rightarrow X = 0$.
\end{itemize}
Therefore, we have the follow result
 
\begin{corollary}
Let $\mathcal{M}_{\mathcal{U}}$ be a von Neumann algebra relative alternative $C^{*}$-algebra $\mathcal{U}$ and $\varphi: \mathcal{M}_{\mathcal{U}} \rightarrow \mathcal{M}_{\mathcal{U}}$ be an additive map. Then the following statements are equivalent:
\begin{enumerate}
\item[$(\spadesuit)$] $\varphi$ is commuting;
\item [$(\clubsuit)$] There exist $Z \in \mathcal{Z}(\mathcal{M}_{\mathcal{U}})$ and an additive map $\Xi : \mathcal{M}_{\mathcal{U}} \rightarrow \mathcal{Z}(\mathcal{M}_{\mathcal{U}})$ such that $\varphi(X) = ZX + \Xi(X)$ for all $X \in \mathcal{M}_{\mathcal{U}}$.
\end{enumerate}

\end{corollary}

\begin{proof}
Since $\mathcal{M}_{\mathcal{U}}$ has no central summands of type $I_1$, then there exists idempotent $e_i \in \mathcal{M}_{\mathcal{U}}$ satisfying $X \mathcal{M}_{\mathcal{U}} \cdot e_i = \{0\} \Rightarrow X = 0$. Now we just need to prove $e_i X e_i \in \mathcal{Z}({\mathcal{M}_{\mathcal{U}}}_{ii})$, there exist $e_j X e_j \in {\mathcal{M}_{\mathcal{U}}}_{jj}$ such that $e_i X e_i + e_j X e_j \in \mathcal{Z}(\mathcal{M}_{\mathcal{U}})$ with $1 \leq i \neq j \leq 2$. 
Let be any $e_i X e_i \in \mathcal{Z}({\mathcal{M}_{\mathcal{U}}}_{ii})$ so by Lemma $4$ in \cite{Bai} there is an element $Z \in \mathcal{Z}(\mathcal{M}_{\mathcal{U}})$ such that $e_i X e_i = Ze_i$. Hence there exist $e_j X e_j = e_j Z e_j \in {\mathcal{M}_{\mathcal{U}}}_{jj}$ such that $e_i X e_i + e_j X e_j \in \mathcal{Z}(\mathcal{M}_{\mathcal{U}})$. Now, by Theorem \ref{mainthm}, the corollary is true.
\end{proof}

\end{document}